\documentclass[10pt]{amsart}
\usepackage{amsmath,amssymb,amsbsy,amsfonts,amsthm,latexsym,
                        amsopn,amstext,amsxtra,euscript,amscd,bm,stmaryrd}
\usepackage{url}
\usepackage[colorlinks,linkcolor=blue,anchorcolor=blue,citecolor=blue]{hyperref}
\usepackage{color}
\usepackage[english]{babel}
  \usepackage{dsfont}
\usepackage{geometry}
\begin{document}

\newtheorem{theorem}{Theorem}
\newtheorem{lemma}[theorem]{Lemma}
\newtheorem{algol}{Algorithm}
\newtheorem{cor}[theorem]{Corollary}
\newtheorem{prop}[theorem]{Proposition}

\newtheorem{proposition}[theorem]{Proposition}
\newtheorem{corollary}[theorem]{Corollary}
\newtheorem{conjecture}[theorem]{Conjecture}
\newtheorem{definition}[theorem]{Definition}
\newtheorem{remark}[theorem]{Remark}

 \numberwithin{equation}{section}
  \numberwithin{theorem}{section}

\newcommand{\comm}[1]{\marginpar{%
\vskip-\baselineskip 
\raggedright\footnotesize
\itshape\hrule\smallskip#1\par\smallskip\hrule}}

\def\sssum{\mathop{\sum\!\sum\!\sum}}
\def\ssum{\mathop{\sum\ldots \sum}}
\def\iint{\mathop{\int\ldots \int}}
\newcommand{\twolinesum}[2]{\sum_{\substack{{\scriptstyle #1}\\
{\scriptstyle #2}}}}

\def\cA{{\mathcal A}}
\def\cB{{\mathcal B}}
\def\cC{{\mathcal C}}
\def\cD{{\mathcal D}}
\def\cE{{\mathcal E}}
\def\cF{{\mathcal F}}
\def\cG{{\mathcal G}}
\def\cH{{\mathcal H}}
\def\cI{{\mathcal I}}
\def\cJ{{\mathcal J}}
\def\cK{{\mathcal K}}
\def\cL{{\mathcal L}}
\def\cM{{\mathcal M}}
\def\cN{{\mathcal N}}
\def\cO{{\mathcal O}}
\def\cP{{\mathcal P}}
\def\cQ{{\mathcal Q}}
\def\cR{{\mathcal R}}
\def\cS{{\mathcal S}}
\def\cT{{\mathcal T}}
\def\cU{{\mathcal U}}
\def\cV{{\mathcal V}}
\def\cW{{\mathcal W}}
\def\cX{{\mathcal X}}
\def\cY{{\mathcal Y}}
\def\cZ{{\mathcal Z}}

\def\C{\mathbb{C}}
\def\F{\mathbb{F}}
\def\K{\mathbb{K}}
\def\Z{\mathbb{Z}}
\def\R{\mathbb{R}}
\def\Q{\mathbb{Q}}
\def\N{\mathbb{N}}
\def\M{\textsf{M}}

\def\({\left(}
\def\){\right)}
\def\[{\left[}
\def\]{\right]}
\def\<{\langle}
\def\>{\rangle}

\def\vec#1{\mathbf{#1}}

\def\e{e}

\def\eq{\e_q}
\def\fS{{\mathfrak S}}

\def\bfalpha{{\boldsymbol \alpha}}

\def\lcm{{\mathrm{lcm}}\,}

\def\fl#1{\left\lfloor#1\right\rfloor}
\def\rf#1{\left\lceil#1\right\rceil}
\def\mand{\qquad\mbox{and}\qquad}

\def\jt{\tilde\jmath}
\def\ellmax{\ell_{\rm max}}
\def\llog{\log\log}

\def\Qbar{\overline{\Q}}

\def\lam{\lambda}
\def\Yildirim{Y{\i}ld{\i}r{\i}m\ }
\def\Lam{\Lambda}
\def\lam{\lambda}

\title[Shifted moments of L-functions and moments of theta functions]
{\bf Shifted moments of L-functions and moments of theta functions}
\author{Marc Munsch}
\address{CRM, Universit\'e de Montr\'eal, 5357 Montr\'eal, Qu\'ebec }
\email{munsch@dms.umontreal.ca}

\date{\today}


\begin{abstract}
 Assuming the Riemann Hypothesis, Soundararajan showed in \cite{SoundRiemann} that $\displaystyle{\int_{0}^{T} \vert \zeta(1/2 + it)\vert^{2k} \ll T(\log T)^{k^2 + \epsilon}}$   . His method was used by Chandee \cite{ChandeeShifts} to obtain upper bounds for shifted moments of the Riemann Zeta function. Building on ideas of \cite{ChandeeShifts} and \cite{SoundRiemann}, we obtain, conditionally, upper bounds for shifted moments of Dirichlet $L$- functions which allow us to derive upper bounds for moments of theta functions.

\end{abstract}

\bibliographystyle{alpha}
\maketitle

\section{Introduction}

For any integer $q$, we denote by  $X_q$ the group of multiplicative characters 
modulo $q$. Denote by  $X_q^+$ the subset of $X_q$ consisting of primitive even characters $\chi$ (those satisfying $\chi(-1) = 1$) and $X_q^-$ the subset of $X_q$ consisting of primitive odd characters  $\chi$ (those satisfying $\chi(-1) = -1$). 
Furthermore, we use $X_q^*$ to denote the set of primitive characters modulo $q$.

\vspace{2mm}

This paper is divided into two parts. Firstly, we study shifted moments of Dirichlet $L$-functions and secondly, we apply this study to obtain upper bounds on moments of theta functions. \vspace{1mm}

A standard problem in analytic number theory is the study moments of the Riemann zeta function or more generally $L$-functions on the critical line. For instance, it is conjectured (see~\cite[Chapter~5]{MM}) that the moments at the central point satisfy the following asymptotic formulas:
\begin{equation}
\label{eq:Conj L}
M_{2k}(q) =\sum_{\chi\in X_q^*}
\vert L(1/2,\chi )\vert^{2k}
\sim C_k q\log^{k^2}q,\hspace{1cm} C_k>0.
\end{equation}  Even though the asymptotic formulas are not known for $k\geq 3$, lower bounds of the expected order of magnitude 
$$\sum_{\chi\in X_q^*}
\vert L(1/2,\chi )\vert^{2k}
\gg q\log^{k^2}q,$$
have been given by Rudnick and  Soundararajan~\cite{RS1} for $q$ prime. Assuming the Generalized Riemann Hypothesis and building on Soundararajan work \cite{SoundRiemann}, we can show that $\displaystyle{M_{2k}(q) \ll q\log^{k^2+\epsilon}q.}$
We can generalize in some way these moments using shifts and consider

\begin{equation} \label{shifts}\sum_{\chi \in X_q^*} \left\vert L\left(\frac{1}{2}+ it_1,\chi\right)\cdots L\left(\frac{1}{2}+ it_{2k},\chi\right)\right\vert \end{equation} where $(t_1,\cdots,t_{2k})$ is a sequence of real numbers. It is expected that if the $t_i$ are reasonably small, we should be able to obtain an asymptotic formula for (\ref{shifts}) (see for instance \cite{Conrey} for a survey about shifted moments in families of $L$- functions). Although, we cannot prove such a general result even assuming GRH, we are able to give a conditional upper bound of nearly the conjectured order of magnitude.

\begin{theorem}\label{shiftedmoments}Assume that the Dirichlet $L$-functions modulo $q$ satisfy the Generalized Riemann Hypothesis. Suppose $q$ is large and the $2k$-tuple $t=(t_1,\cdots, t_{2k})$ is such that $t_i \ll \log q$. Then, for all $\epsilon>0$, we have the uniform bound:

$$ \sum_{\chi \in X_q^*} \left\vert L\left(\frac{1}{2}+it_1,\chi\right)\cdots L\left(\frac{1}{2}+it_{2k},\chi\right)\right\vert \ll \phi(q)(\log q)^{k/2+\epsilon} \prod_{i<j} E_{i,j}$$where

$$E_{i,j}=\begin{cases}\left(\min\left\{\frac{1}{\vert t_i-t_j\vert}, \log q\right\}\right)^{1/2}  &\mbox{if   } \vert t_i-t_j\vert \leq \frac{1}{100}, \\
\sqrt{\log \log q}&\mbox{if   }\vert t_i-t_j\vert \geq \frac{1}{100}.\\ \end{cases}$$

\end{theorem} This can be related to the main result of \cite{ChandeeShifts} and enlightens the fact that $L\left(\frac{1}{2}+it_i,\chi\right)$ and $L\left(\frac{1}{2}+it_j,\chi\right)$ are essentially correlated when $\vert t_i -t_j\vert \approx \frac{1}{\log q}$ and ``independent" as long as $\vert t_i - t_j \vert$ is significantly larger than $1/\log q$.

\vspace{2mm}

For real $x> 0$ and $\eta\in\{ 0,1\}$ we set 
$$\theta (\eta, x,\chi)
=\sum_{n=1}^\infty \chi (n)n^{\eta}e^{-\pi n^2x/q}, 
\qquad \chi\in X_q. 
$$ We note that, if we set 
$\eta_{\chi}=1$ if $\chi$ is odd and $\eta_{\chi}=0$ otherwise,
then 
$$\theta_{q} (\eta_{\chi}, x,\chi)=\theta_{q} (x,\chi)
$$
is the classical theta-function of the character $\chi$ (see~\cite{Dav} for a background and basic properties). We can express these values using Mellin transforms of $L$- functions which make the use of our result about moments of shifted $L$- functions very appropriate.
\vspace{2mm}

 When computing the root number of $\chi$ appearing in the functional equation of the associated Dirichlet $L$- function, the question of whether $\theta_{q} (1,\chi) \neq 0$ appears naturally (see~\cite{Lou99} for details). Numerical computations led to the conjecture that it never happens if $\chi$ is primitive (see~\cite{CZ} for a counterexample with $\chi$ unprimitive). In order to investigate the non-vanishing of theta functions at their central point, the study of moments has been initiated in \cite{LoMu}, \cite{LoMu1} and \cite{MS}. Let us define 
$$S_{2k}^+(q) =\sum_{\chi\in X_q^+\backslash \chi{_0}}\vert\theta (1,\chi)\vert^{2k} \mand S_{2k}^-(q) =\sum_{\chi\in X_q^-}\vert\theta (1,\chi)\vert^{2k}.$$  It is conjectured in \cite{MS}, based on numerical computation and some theoretical support, that 
\begin{equation}
\label{eq:Conj TS}
\begin{split}
&S_{2k}^+(q)\sim a_k\phi(q)q^{k/2}\(\log q\)^{(k-1)^2},\\
& S_{2k}^-(q)\sim b_k\phi(q)q^{3k/2}\(\log q\)^{(k-1)^2}
\end{split}
\end{equation} 
for some positive constants $a_k$ and $b_k$, depending only on $k$.  Recently, a lower bound of expected order for $S_{2k}^+(q)$ and $S_{2k}^-(q)$ has been proven unconditionally in \cite{MS}. In the second part of the paper, we will derive upper bounds giving good support towards Conjecture (\ref{eq:Conj TS}). \newpage \noindent Precisely, we prove

\begin{theorem}\label{upperGRH} Assume the Generalized Riemann Hypothesis for all Dirichlet $L$- functions modulo $q$. Then, for all $\epsilon>0$, we have%
$$S_{2k}^+(q) \ll \phi(q)q^{k/2}(\log q)^{(k-1)^2+\epsilon} \mand S_{2k}^-(q) \ll \phi(q)q^{3k/2}(\log q)^{(k-1)^2+\epsilon}.$$  

  
   \end{theorem} 
   
   \vspace{1mm}
 This can be related to recent results of~\cite{HarperMaks} (see also~\cite{Heap}), where the authors obtain the asymptotic behaviour of a Steinhaus random multiplicative function (basically a multiplicative random variable whose values at prime integers are uniformly distributed on the unit circle). This can be viewed as a random model for $\theta_{q} (x,\chi)$. In fact, the rapidly decaying factor $e^{-\pi n^2/q}$ is mostly equivalent to restrict the sum over integers $n \le n_0(q)$ for some $n_0(q) \approx \sqrt{q}$ and the averaging behavior of $\chi(n)$ with $n\ll q^{1/2}$ is essentially similar to that of a Steinhaus random multiplicative function. Hence, these results are a good support for Conjecture~\eqref{eq:Conj TS}. Upper bounds of Theorem \ref{upperGRH} together with lower bounds obtained in \cite[Theorem $1.1$]{MS} confirm this heuristic.

\vspace{1mm}
The method of the proof of Theorem \ref{upperGRH} relies on the bound obtained for moments of shifted $L$- functions.

\section{Moments of shifted L-functions}

 In that section, we mostly adapt results and ideas of \cite{SoundRiemann} to our situation. These techniques build on ideas of Selberg about the distribution of $\vert \log \zeta(1/2+it)\vert$ (see \cite{Selberg}).The starting point is the following equality
 
 $$ \int_{T}^{2T} \vert \zeta(1/2 + it)\vert^{2k}  dt = -\int_{-\infty}^{+\infty} e^{2kV} d \text{ meas}(S(T,V))=2k\int_{-\infty}^{+\infty} e^{2kV} \text{meas}(S(T,V)) dV$$ where $S(T,V)=\left\{t\in \left[T,2T\right]: \log \vert \zeta(1/2+it)\vert \geq V\right\}$. From this, an upper bound for the moment can be directly deduced from the upper bound of $\text{meas } (S(T,V))$. In our case, we have to study the frequency (in terms of characters) of large values of $L$-functions. Thus, we will proceed in the same way by estimating the measure of  
$$S_t(q,V)=\left\{\chi (\bmod \,q), \chi^2 \neq \chi_0,: \log\left\vert L\left(\frac{1}{2}+it_1,\chi\right)\right\vert+\cdots+\log\left\vert L\left(\frac{1}{2}+it_{2k},\chi\right)\right\vert \geq V\right\}$$ for $V>0$ and a $2k$-tuple $t=(t_1,\cdots, t_{2k})$. Most of the work consists in keeping the dependence both in terms of the modulus $q$ and the height of the shifts. If the shifts are not too large, we are able to obtain a quasi-optimal upper bound under the Generalized Riemann Hypothesis. This result will be sufficient for our application to moments of theta functions. It should be noticed that the recent method of Harper (see \cite{Harper}) may be used to remove the $\epsilon$ factor in Theorem \ref{shiftedmoments}.

\vspace{5mm}

 Let set $N_t(q,V)= \# S_t(q,V)$. We can express the shifted moments of $L$- functions as the following

\begin{align}\label{momentshifted}
\sum_{\chi \in X_q^*}\left\vert L\left(\frac{1}{2}+it_1,\chi\right)\cdots L\left(\frac{1}{2}+it_{2k},\chi\right)\right\vert = \sum_{\chi \in X_q^*} e^{\log\left\vert L(\frac{1}{2}+it_1,\chi)\right\vert+\cdots+\log\left\vert L(\frac{1}{2}+it_{2k},\chi)\right\vert}  \notag \\
=\sum_{\chi \in X_q^*}\int_{-\infty}^{\log\left\vert L(\frac{1}{2}+it_1,\chi)\right\vert+\cdots+\log\left\vert L(\frac{1}{2}+it_{2k},\chi)\right\vert} e^{V}dV=\int_{-\infty}^{+\infty} e^{V} N_t(q,V)dV + q^{o(1)}.
\end{align} The error term comes from the contribution of quadratic characters which can easily be bounded, using Corollary \ref{upperboundline} by
$$O\left(\int_{-\infty}^{4ck\frac{\log q}{\log \log q}}e^V dV \right) \ll q^{o(1)}.$$  Hence, the problem of estimating the moments boils down to getting precise bounds for $N_t(q,V)$. In order to do that, let us define the following quantity

$$W=2k \log\log q + 2\sum_{i,j \atop i<j}F_{i,j}$$ where

$$F_{i,j}=\begin{cases}\log\left(\min\left\{\frac{1}{\vert t_i-t_j\vert}, \log q\right\}\right)  &\mbox{if   } \vert t_i-t_j\vert \leq \frac{1}{100}, \\
\log \log \log q&\mbox{if   }\vert t_i-t_j\vert \geq \frac{1}{100}.\\ \end{cases}$$
 
  We will prove the following theorem which estimates the measure of $S_t(q,V)$ for large $q$ and all $V$.

\begin{theorem}\label{largevalues}Assume that the Dirichlet $L$-functions modulo $q$ satisfy the Generalized Riemann Hypothesis. Suppose that $\vert t\vert\leq T \leq \log^A q$ where $A>0$ and $V$ is a large real number. If $4\sqrt{\log \log q}\leq V\leq W$ then

$$N_t(q,V) \ll \phi(q)\frac{V}{\sqrt{W}}\exp\left(-\frac{V^2}{W}\left(1-\frac{18k}{5\log W}\right)^2\right);$$ if $W<V<\frac{1}{4k}W\log W$ we have
$$N_t(q,V)\ll \phi(q)\frac{V}{\sqrt{W}}\exp\left(-\frac{V^2}{W}\left(1-\frac{18kV}{5W\log W}\right)^2\right);    $$ and if $\frac{1}{4k}W\log W<V$ we have 
$$N_t(q,V)\ll \phi(q)\exp\left(-\frac{V}{801k}\log V\right)   .$$
\end{theorem}

\vspace{1mm}
\textit{\bf{Proof of Theorem \ref{shiftedmoments}}} Inserting the bounds of Theorem \ref{largevalues} in Equation (\ref{momentshifted}) gives the upper bound in Theorem \ref{shiftedmoments}. Precisely, it is appropriate for this computation to use Theorem \ref{largevalues} in the weakest form 
 
 $$N_t(q,V) \ll \phi(q)(\log q)^{o(1)} exp(-V^2/W) \text{  for }  3\leq V\leq 200W, $$
$$N_t(q,V) \ll \phi(q)(\log q)^{o(1)} exp(-2V) \text{  for }  V> 200W .$$ This allows us to bound the moments by $\phi(q)(\log q)^{o(1)}e^{W/4}$ which concludes the proof.

\subsection{Preliminary results}

We regroup in that subsection all the technical results that we will use in the proof of Theorem \ref{largevalues}. These are mainly suitable adaptations to our case of Lemmas of \cite{SoundRiemann}. In the sequel, we will always write $s=\sigma + it$ for a complex number $s$. We write $\log^{+}(x):=\max (\log x,0)$.

\begin{lemma} \label{upsigma}
Unconditionally, for any $s$ not coinciding with $1$, $0$ or a zero of $L(s,\chi)$, and for any $x \geq 2,$ we have
\begin{eqnarray*}
-\frac{L'}{L}(s,\chi) &=& \sum_{n \leq x} \frac{\chi(n)\Lam(n)}{n^{s}}\frac{\log\frac{x}{n}}{\log x} + \frac{1}{\log x}\left( \frac{L'}{L}(s,\chi) \right)' + \frac{1}{\log x} \sum_{\rho \neq 0,1} \frac{x^{\rho - s}}{(\rho - s)^2} \\
&& + \frac{1}{\log x} \sum_{n=0}^{\infty} \frac{x^{-2n-a-s}}{(2n + a +s)^2}.
\end{eqnarray*}
\end{lemma}
 
\begin{proof}This is Lemma $2.4$ of \cite{ChandeeCritical} with $a(n)=\chi(n)\Lam(n)$, $d=1$ and $k(j)=a$ (here $a = 0$ or $1$ is the number given by $\chi(-1) = (-1)^a$). \end{proof}

\begin{proposition}\label{large values} Assume GRH for all Dirichlet L-functions of modulus $p$. Let $T$ be a parameter and let $x\ge 2$.  Let $\lambda_0=0.56\ldots$ 
denote the unique positive real number satisfying $e^{-\lambda_0} = \lambda_0$.   
For all $\lambda \ge \lambda_0$, the following estimate 
$$
\log |L(\sigma + it,\chi)| \le \text{\rm Re } \sum_{n\le x} \frac{\chi(n)\Lam(n)}{n^{\frac{1}{2}+ \frac{\lam}{\log x} +it} \log n} 
\frac{\log (x/n)}{\log x} + \frac{(1+\lam)}{2} \frac{\log (q)+\log^{+}(T)}{\log x} + O\Big( \frac{1}{\log x}\Big)
$$ holds uniformly for $\vert t\vert \leq T$ and $1/2\leq \sigma\leq \sigma_0=1/2 + \frac{\lambda}{\log x}$.

\end{proposition}

\begin{proof} Let $a = 0$ or $1$ be again the number given by $\chi(-1) = (-1)^a$.
Letting $\rho=1/2 +i\gamma$ run over the non-trivial 
 zeros of $L(s,\chi)$, we define  
 $$
 F_{\chi}(s) = \text{Re }\sum_{\rho} \frac{1}{s-\rho} = \sum_{\rho} \frac{\sigma-1/2}{(\sigma-1/2)^2+ 
 (t-\gamma)^2}. 
 $$Obviously $F_{\chi}(s)$ is non-negative in the half-plane $\sigma \ge 1/2$. By Hadamard's factorization (see \cite[Chapter 12, Eq. (17)]{Dav}), we have \begin{equation}\label{Hadamard} \frac{L'(s,\chi)}{L(s,\chi)}=-\frac{1}{2}\log \frac{q}{\pi}-\frac{1}{2}\frac{\Gamma^{'}}{\Gamma}\left(\frac{s+a}{2}\right)+B(\chi)+\sum_{\rho} \left(\frac{1}{s-\rho} + \frac{1}{\rho}\right).\end{equation} Here, $B(\chi)$ is a constant depending only on $\chi$, whose real part is given by 
 
 $$\text{Re }(B(\chi))=-\sum_{\rho}\frac{1}{\rho}.$$ By taking the real parts of both sides of (\ref{Hadamard}), an application of Stirling's formula yields 
 
 \begin{equation}\label{ineqReal}    
 -\text{Re }\frac{L'(s,\chi)}{L(s,\chi)}=\frac{\log(q)+\log^{+}(t)}{2} -F_{\chi}(s) + O(1) \leq \frac{\log(q)+\log^{+}(T)}{2}+O(1)
 \end{equation} where we used the positivity of $F_{\chi}(s)$ in that region. Integrating (\ref{ineqReal}) as $\sigma=\text{Re }(s)$ varies from $\sigma$ to $\sigma_0(>1/2)$, we obtain, setting $s_0=\sigma_0 +it$,
 
 \begin{equation}\label{deuxabscisses} \log\vert L(s,\chi)\vert -\log\vert L(s_0,\chi)\vert \leq \left(\frac{\log(q)+\log^{+}(T)}{2}+O(1)\right)(\sigma_0 - \sigma).\end{equation} On the other hand, using Lemma \ref{upsigma}, we get 
 
 \begin{align}\label{logderivative} -\frac{L'}{L}(s,\chi) &= &\sum_{n \leq x} \frac{\chi(n)\Lam(n)}{n^{s}}\frac{\log\frac{x}{n}}{\log x} + \frac{1}{\log x}\left( \frac{L'}{L}(s,\chi) \right)' + \frac{1}{\log x} \sum_{\rho \neq 0,1} \frac{x^{\rho - s}}{(\rho - s)^2} \\
&&+\frac{1}{\log x} \sum_{n=0}^{\infty} \frac{x^{-2n-a-s}}{(2n + a +s)^2}  \end{align} for any $s$ not coinciding with a zero of $L(s,\chi)$ and for any $x \geq 2$. Taking $s=\sigma+it$, integrating (\ref{logderivative}) over $\sigma$ from $\sigma_0$ to $\infty$ and extracting the real parts, we have, for $x \geq 2$,

 \begin{align}\label{integrate}
\log |L(s_0,\chi)| = \text{Re }\Big( \sum_{2\le n\le x} \frac{\chi(n)\Lambda(n)}{n^{s_0} \log n} \frac{\log (x/n)}{\log x}  
  &- \frac{1}{\log x} \frac{L^{\prime}}{L}(s_0,\chi) \notag \\
  & \hspace{-8mm} + \frac{1}{\log x} \sum_{\rho} 
 \int_{\sigma_0}^{\infty} \frac{x^{\rho-s}}{(\rho-s)^2} d\sigma +O\Big(\frac{1}{\log x}\Big)\Big).
 \\ \notag
 \end{align} The integral in (\ref{integrate}) is bounded as follows:
 
 $$
\sum_{\rho}\Big|\int_{\sigma_0}^{\infty} \frac{x^{\rho -s}}{(\rho -s)^2} d\sigma\Big| 
\le \sum_{\rho}\int_{\sigma_0}^{\infty}\frac{ x^{\frac 12-\sigma}}{|s_0-\rho|^2} d\sigma 
= \sum_{\rho}\frac{x^{\frac 12-\sigma_0}}{|s_0-\rho|^2 \log x}= \frac{x^{\frac 12-\sigma_0}F_{\chi}(s_0)}{(\sigma_0-\frac 12)\log x}.$$ Thus, using (\ref{ineqReal}), we deduce that for $x\geq 2$

\begin{align}\label{majorations0}\log |L(s_0,\chi)| &\leq \text{Re } \sum_{2\le n\le x} \frac{\chi(n)\Lambda(n)}{n^{s_0} \log n} \frac{\log (x/n)}{\log x}   \\
 &+ \frac{\log(q)+\log^+(t)}{2\log x}
 - \frac{F_{\chi}(s_0)}{\log x}
 + \frac{x^{\frac 12-\sigma_0}}{\log^2 x}\frac{F_{\chi}(s_0)}{(\sigma_0-\frac 12)} \notag
 + O\left(\frac{1}{\log x}\right). \notag \\ \notag
\end{align} Hence, combining (\ref{deuxabscisses}) together with (\ref{majorations0}), the following inequality 

\begin{align}\label{majlog}
 \log |L(\sigma+ it,\chi)|  &\le \frac{\log(q)+\log^{+}(T)}{2}\left(\sigma_0 - \sigma + \frac{1}{\log x}\right)
+ \text{Re }\sum_{2 \le n\le x} \frac{\chi(n)\Lambda(n)}{n^{s_0} \log n}  \frac{\log (x/n)}{\log x} \notag \\
 &\hskip .5 in +F_{\chi}(s_0) \Big( \frac{x^{\frac 12-\sigma_0}}{(\sigma_0-\frac 12) \log^2 x} -\frac{1}{\log x} \Big) + O\Big(\frac{1}{\log x}\Big) 
\end{align} holds for $x\geq 2$ and uniformly for $1/2\leq \sigma \leq \sigma_0 \leq 3/2$, $\vert t\vert \leq T$. We choose $\sigma_0 =\frac12 + \frac{\lam}{\log x}$, where $\lam \ge \lam_0$. This restriction on $\lam$ ensures that the term involving $F_{\chi}(s_0)$ in (\ref{majlog}) makes a negative contribution 
 and may therefore be omitted. The proposition follows easily.  

\end{proof}

\begin{corollary}\label{upperboundline} Let $\chi$ be a primitive character modulo $q$ and assume GRH for $L(s,\chi)$. Then if $q$ is large enough, there exists an absolute constant $c>0$ such that

$$ \left\vert L\left(\frac{1}{2}+it,\chi\right)\right\vert \ll \exp\left(c \frac{\log q + \log^{+} t}{\log \log q}\right).$$
\end{corollary}

\begin{proof} This follows directly from the above proposition by setting $x=\log^{2-\epsilon} q$.

\end{proof}

\begin{remark} This inequality is less precise than \cite[Corollary 1.2]{ChandeeCritical} when $qt$ is large. Nevertheless, this covers the case when $t$ is relatively small compared to $q$ which is suitable for our applications.\end{remark}

Our proof of Theorem \ref{shiftedmoments} rests upon our main Proposition \ref{large values}. We begin by showing that the sum over prime powers appearing in that proposition may be in fact restricted over primes.

\begin{lemma}\label{primepowers}
Assume that the Dirichlet $L$-functions modulo $q$ satisfy the Generalized Riemann Hypothesis. Let $t \le \log^A q$ with $A>0$, $x\ge 2$ and $\sigma\ge \frac 12$.  Then, if $\chi$ is a Dirichlet character modulo $q$ such that $\chi \neq \chi_0^{2}$, we have
$$
\left| \sum_{n\le x \atop n \neq p} \frac{\chi(n)\Lambda(n)}{n^{\sigma+it}\log n}  
\frac{\log x/n}{\log x} \right| \ll \log \log\log q +O(1).
$$

\end{lemma}

\begin{proof}
Clearly, the contribution coming from the prime powers $p^k$ with $k\geq 3$ is $\ll 1$. It remains to handle the terms $n=p^2$. Hence, we have to bound

\begin{equation}\label{partial}\sum_{p\leq \sqrt{x}} \frac{\chi^2(p)}{p^{2\sigma + 2it}}\frac{\log(\sqrt{x}/p)}{\log \sqrt{x}}.\end{equation} We split this sum into ranges $2\leq p \leq \log^{8+4A+\epsilon} q$ and $\log^{8+4A+\epsilon} q \le p \le \sqrt{x}$. Then the first sum is easily bounded by $\displaystyle{\sum_{p\leq \leq \log^{8+4A+\epsilon} q} 1/p \ll \log\log\log q}$.

\vspace{2mm}
To treat the second sum, let us recall (see for instance \cite[p. 125]{Dav}) that under GRH, the estimate 
 
 $$ \sum_{n\leq x} \chi(n)\Lambda(n) \ll x^{1/2}\log^2 (qx)$$ holds for $x\ge 2$ and $\chi$ a non trivial character.  By partial summation, we can deduce that 
 $$ \sum_{p\le x}\frac{\chi(p)\log p}{p^{2it}} \ll \vert t\vert x^{1/2}\log^2 (qx).$$ Thus, again by partial summation, we derive (using our restriction on $t$ and the fact that $\chi^2$ is non trivial) that the sum over primes $\geq \log^{8+4A+\epsilon} q $ is $O(1)$, which concludes the proof.

\end{proof}

Proposition \ref{large values} together with Lemma \ref{primepowers} give directly
\begin{corollary}\label{boundprimes}
For a Dirichlet character $\chi$ modulo $q$ such that $\chi^2 \neq \chi_0$, the inequality

$$\log |L(\sigma +it,\chi)| \le \text{\rm Re } \sum_{p\le x} \frac{\chi(p)}{p^{\sigma+ \frac{\lam}{\log x} +it}} 
\frac{\log (x/p)}{\log x} + \frac{(1+\lam)}{2} \frac{\log q + \log^{+} T}{\log x} + O\Big(\log \log \log q\Big)$$ holds uniformly for $\vert t\vert \leq T<\log^A q$ and $1/2\leq \sigma \leq 1/2+ \frac{\lam}{\log x}$.
\end{corollary}

The next lemma is a $q$-analogue of \cite[Lemma $3$]{SoundRiemann}.

\begin{lemma}\label{momentsprimes}
Suppose $x\geq 2$ and $k$ is an integer such that $x^k<q$. Then for any $t\in\mathbb{R}$ and any complex numbers $a(p)$ we have

$$  \sum_{\chi \in X_q}\left|\sum_{p\leq x}\frac{\chi(p)a(p)}{p^{1/2+it}}\right|^{2k}\leq \phi(q)k!\left(\sum_{p\leq x}\frac{|a(p)|^2}{p}\right)^{k}.$$ Hence, there exist positive constants $c_{\chi}$ such that $\sum_{\chi \bmod q}c_{\chi}=\phi(q)$ and the following inequality holds:

$$ \left|\sum_{p\leq x}\frac{\chi(p)a(p)}{p^{1/2+it}}\right|^{2k}\leq c_{\chi}k!\left(\sum_{p\leq x}\frac{|a(p)|^2}{p}\right)^{k}.$$
\end{lemma}

\begin{proof}The proof follows the same lines as in the proof of Lemma $3$ of \cite{SoundRiemann}. After expanding the $2k$-th power, we use the orthogonality of characters modulo $q$ (here the inequality $x^k<q$ ensures that $m=n \bmod\, q$ implies $m=n$) instead of the orthogonality in $t$-aspect.

\end{proof}

We will need the following adaptation of \cite[Lemma $3.5$]{ChandeeShifts}.

\begin{lemma}\label{cos}

$$\sum_{p\leq z} \frac{\cos(a\log p)}{p} \leq \begin{cases} \log\left(\min\left\{\frac{1}{\vert a\vert},\log z\right\}\right)+ O(1) &\mbox{if   } \vert a\vert \leq \frac{1}{100}, \\ 
\log\log(2+\vert a\vert)+O(1)& \mbox{if   } \vert a\vert\geq \frac{1}{100}. \\\end{cases}$$

\end{lemma}

\begin{proof}  
If $\vert a\vert \leq \frac{1}{\log z}$, it follows from Mertens' Theorem. Otherwise, we use inequality $(2.1.6)$, p.$57$ of \cite{GS}.
\end{proof}

\subsection{Proof of Theorem \ref{largevalues}}

First, remark that if $-\infty\leq V\leq 4\sqrt{\log \log q}$, then trivially we have 

$$\int_{-\infty}^{+\infty} e^{V} N_t(q,V)dV \leq \phi(q) e^{4\sqrt{\log \log q}} = o(\phi(q)\log q).$$ In view of Corollary \ref{upperboundline}, we can assume $4\sqrt{\log \log q} \leq V \leq 4ck \frac{\log q}{\log \log q}$ using the fact that $t_0=\max(t_i, i=1\cdots, 2k) \leq \log^A q.$ It remains to estimate $N_t(q,V)$ for large $q$ with an explicit dependence on $t_0$. Choosing $\lambda=0.6$ in Corollary \ref{boundprimes}, we obtain if $\chi^2 \neq \chi_0$ that 

\begin{flalign}\label{somme}
&\log\left\vert L\left(\frac{1}{2}+it_1,\chi\right)\right\vert+\cdots+\log\left\vert L\left(\frac{1}{2}+it_{2k},\chi\right)\right\vert & \notag  \\ \notag
&\leq  \text{\rm Re } \left(\sum_{p\le x} \frac{\chi(p)p^{-it_1}}{p^{\frac{1}{2}+ \frac{0.6}{\log x}}} 
\frac{\log (x/p)}{\log x}+\cdots+ \frac{\chi(p)p^{-it_{2k}}}{p^{\frac{1}{2}+ \frac{0.6}{\log x}}} 
\frac{\log (x/p)}{\log x}\right)+ \frac{8k}{5} \frac{\log q + \log^+ T}{\log x}& \\  \notag
&+ O\Big(\log \log \log q\Big). &\\ \notag
\end{flalign} Following \cite{ChandeeShifts} and \cite{SoundRiemann}, we define the quantity $A$ as 

$$ A = \begin{cases}\frac{\log W}{2} &\mbox{if   } 4\sqrt{\log \log q}\leq V\leq W, \\ 
\frac{W\log W}{2V}& \mbox{if   } W\leq V\leq \frac{1}{4k}W\log W, \\
2k & \mbox{if   } V>\frac{1}{4k}W\log W.\end{cases} 
$$Let $x=(q\max(T,1))^{A/V}$ and $z=x^{1/\log \log q}$. From the previous bounds, we have

\begin{flalign*}&\log\left\vert L\left(\frac{1}{2}+it_1,\chi\right)\right\vert+\cdots+\log\left\vert L\left(\frac{1}{2}+it_{2k},\chi\right)\right\vert & \\
&\leq S_1(\chi)+S_2(\chi)+\frac{8k}{5} \frac{\log q + \log^{+} T}{\log x}+O\Big(\log \log \log q\Big),&\\
\end{flalign*}
where

$$  S_1(\chi)=\left\vert\sum_{p\le z} \frac{\chi(p)(p^{-it_1}+\cdots+p^{-it_{2k}})}{p^{\frac{1}{2}+ \frac{0.6}{\log x}}}\frac{\log (x/p)}{\log x}\right\vert$$ and

$$ S_2(\chi)=\left\vert\sum_{z< p\le x} \frac{\chi(p)(p^{-it_1}+\cdots+p^{-it_{2k}})}{p^{\frac{1}{2}+ \frac{0.6}{\log x}}}\frac{\log (x/p)}{\log x}\right\vert .$$ It remains to study how often with respect to characters these quantities could be large.
Firstly, if $\chi \in S_t(q,V)$, we must have \footnote{Compare $\frac{V}{A}$ to $\log\log\log q$.} 

$$  S_1(\chi) \geq V_1:=V\left(1-\frac{9k}{5A}\right)\text{    or   } S_2(\chi)\geq V_2:=\frac{kV}{5A}.$$ Let
$N_i(q)=\# S_i(q):= \left\{ \chi (\bmod \,q), \chi^2 \neq \chi_0 : S_i(\chi)\geq V_i\right\}$ for $i=1,2$. We want to find upper bounds for $N_i(q)$ with a certain uniformity in $t$ \footnote{We have to keep in mind that for our applications $t$ will be at most of size $\log q$.}. By Lemma \ref{momentsprimes}, we see that, for any natural number $l\leq \frac{3}{4}\frac{V}{A}$, \footnote{Here we use that $T\leq \log^A q$.} we have 

$$ \left\vert S_2(\chi)\right\vert^{2l} \leq c_{\chi}l!\left(\sum_{z< p\le x} \frac{4k^2}{p} \right)^{l} \ll c_{\chi}(4lk^2(\log\log\log q+ O(1))^{l}.$$ Choosing $l=\lfloor 3V/4A \rfloor$ and observing that 

\begin{equation}\label{N2} \sum_{\chi \in S_2(q)}\left\vert S_2(\chi)\right\vert^{2l}  \geq N_2(q)V_2^{2l}\end{equation} we derive 

\begin{equation*} N_2(q) \ll \sum_{\chi \in S_2(q)} c_{\chi} \left(\frac{5A}{kV}\right)^{2l}(4lk^2(\log\log\log q+ O(1))^{l} \ll \phi(q) \exp\left(-\frac{V}{2A}\log V\right). \end{equation*} It remains to find an upper bound for $N_1(q)$. By Lemma \ref{momentsprimes}, for any $l < \frac{\log q}{\log z}$,

\begin{align*} \left\vert S_1(\chi)\right\vert^{2l}& \leq  c_{\chi}l!\left(\sum_{ p\le z} \frac{\vert p^{-it_1}+\cdots+p^{-it_{2k}}\vert^2}{p} \right)^{l} \\
&\ll c_{\chi}l!\left(\sum_{ p\le z} \frac{2k+2\sum_{i<j}\cos((t_i-t_j)\log p)}{p} \right)^{l} \\
& \ll c_{\chi}l!\left(2k\log\log z+2\sum_{i,j\atop i<j}F_{i,j}\right)^l \\
&\ll c_{\chi}l! \,W^l \ll c_{\chi}\sqrt{l}\left(\frac{lW}{e}\right)^l \\
\end{align*} where we applied Lemma \ref{cos} and used together Stirling's formula and the fact that $z<q$. 

\begin{remark} If $\vert t_i-t_j\vert \geq \frac{1}{100}$, by hypothesis this quantity is at most $2\log^A q$. Hence, the second case of Lemma \ref{cos} implies that the sum over primes is $\ll \log\log\log q$. \end{remark}Proceeding as in (\ref{N2}) for $N_2$, we deduce that 

$$ N_1(q) \ll V_1^{-2l} \sum_{\chi \in S_1(q)}\left\vert S_1(\chi)\right\vert^{2l} \ll \phi(q)\sqrt{l}\left(l\frac{W}{eV_1^2}\right)^l. $$
When $V\leq \frac{W^2}{4k^3}$, we choose $l=\lfloor\frac{V_1^2}{W}\rfloor$, and when $V>\frac{W^2}{4k^3}$, we choose $l=\lfloor 8V\rfloor$. We easily verify, using the definition of $A$, that the condition $l<\frac{\log q}{\log z}$\footnote{Use the fact that $W/4k^2 \leq \log\log q.$} holds in both cases.  Finally we get

$$N_1(q) \ll \phi(q)\frac{V}{\sqrt{W}}\exp\left(-\frac{V_1^2}{W}\right)+\phi(q)\exp(-3V\log V).$$
 Using our bounds on $N_1(q)$ and $N_2(q)$, elementary computations lead to the proof of Theorem \ref{largevalues}.

\section{Application to upper bounds for moments of theta functions}

 In that section, we will prove Theorem \ref{upperGRH} in the case of even characters. The proof for odd characters goes exactly along the same lines. The method is the following, we express theta values as Mellin transform of $L$- functions and then we use our previous result about moments of shifted $L$- functions.

\vspace{1mm}
For every even primitive character $\chi$ modulo $q$, recall the following relation for $c>1/2$

$$\theta(1,\chi)=\int_{c-i\infty}^{c+\infty} L(2s,\chi)\left(\frac{q}{2\pi}\right)^{s}\Gamma(2s)ds.$$

Shifting the line of integration to $\Re(s)=1/4$ and using the decay of $\Gamma(s)$ in vertical strips, we end up with 

$$\theta(1,\chi)=\left(\frac{q}{\pi}\right)^{\frac{1}{4}}\int_{-\infty}^{\infty}L\left(\frac{1}{2}+2it,\chi\right)\left(\frac{q}{\pi}\right)^{2it}\Gamma\left(\frac{1}{2}+2it\right)dt.$$

We express the moments as

\begin{equation}\label{moments2k}\sum_{\chi \in X_q^+ \backslash\chi_0}\vert \theta(1,\chi)\vert^{2k}=\left(\frac{q}{\pi}\right)^{\frac{k}{2}} \sum_{\chi \in X_q^+\backslash\chi_0}\left\vert \int_{-\infty}^{\infty}L\left(\frac{1}{2}+2it,\chi\right)\left(\frac{q}{\pi}\right)^{2it}\Gamma\left(\frac{1}{2}+2it\right)dt\right\vert^{2k}.\end{equation} Hence, the problem boils down to getting a bound of size $\log^{(k-1)^2+\epsilon} q$ for the $2k$-fold integral. In the following, we can sum over $X_q^*$ without substantial loss.

\subsection{Cutting part}

 The strategy is the following: we will cut up to a certain reasonable height, for instance $\log^{\epsilon} q$. Precisely, using the decay of $\Gamma\left(\frac{1}{2}+2it\right)$, we bound the tail:

\begin{lemma}\label{resteintegral} Fix $\epsilon>0$. There exists an absolute constant $c$ such that
$$\sum_{\chi \in X_q^*}\left\vert \int_{-\infty}^{\infty}L\left(\frac{1}{2}+2it,\chi\right)\left(\frac{q}{\pi}\right)^{2it}\Gamma\left(\frac{1}{2}+2it\right)\mathds{1}_{\vert t\vert \geq\log^{\epsilon}(q)}(t)dt\right\vert^{2k} \ll \phi(q) e^{-c\log^{\epsilon}q}.$$
\end{lemma}

\begin{proof}

Using H\"older inequality with parameters $\frac{1}{2k}+\frac{2k-1}{2k}=1$, the problem reduces to bound

$$\sum_{\chi \in X_q^*}\left(\int_{\vert t\vert \geq\log^{\epsilon} q}\left\vert L\left(\frac{1}{2}+2it,\chi\right)\right\vert^{2k}\left\vert\Gamma\left(\frac{1}{2}+2it\right)\right\vert dt\right)\left(\int_{\vert t\vert \geq\log^{\epsilon} q}\left\vert \Gamma\left(\frac{1}{2}+2it\right)\right\vert dt\right)^{2k-1}.$$ We decompose dyadically the range of integration in the left hand side and use the convergence of the right hand side to end up with

$$\sum_{n\geq \log^{\epsilon} q} \sum_{\chi \in X_q^*}\int_{n}^{2n}\left\vert L\left(\frac{1}{2}+2it,\chi\right)\right\vert^{2k}\left\vert\Gamma\left(\frac{1}{2}+2it\right)\right\vert dt.$$Using Stirling's formula and Proposition $2.9$ of \cite{LiChandee}\footnote{The method is the same as our proof of Theorem \ref{shiftedmoments}.}, we get for $c_1>0$ an absolute constant

$$\begin{array}{l} \displaystyle{\sum_{n\geq \log^{\epsilon} q} \sum_{\chi \in X_q^*}\int_{n}^{2n}\left\vert L\left(\frac{1}{2}+2it,\chi\right)\right\vert^{2k}\left\vert\Gamma\left(\frac{1}{2}+2it\right)\right\vert dt}  \\
\displaystyle{\ll  \sum_{n\geq \log^{\epsilon} q} e^{-c_1n} \sum_{\chi \in X_q^*}\int_{n}^{2n}\left\vert L\left(\frac{1}{2}+2it,\chi\right)\right\vert^{2k} dt} \\
\displaystyle{\ll \phi(q)(\log q)^{k^2+\epsilon} \sum_{n\geq \log^{\epsilon} q} e^{-c_1n} n(\log n)^{k^2+\epsilon} \ll \phi(q) e^{-c\log^{\epsilon}q}.}
\end{array}$$

\end{proof}

\subsection{Bound for the hypercube integral}

 It remains to bound optimally the integral on the $2k$-hypercube $\mathcal{H}$ of size $\log^{\epsilon} q$. First, observing that $\Gamma\left(\frac{1}{2}+2it\right)$ is bounded on $\mathcal{H}$ and expanding the integral in (\ref{moments2k}), we get 
 
\begin{eqnarray}\label{hypercube} \sum_{\chi \in X_q^*}\left\vert \int_{-\infty}^{\infty}L\left(\frac{1}{2}+2it,\chi\right)\left(\frac{q}{\pi}\right)^{2it}\Gamma\left(\frac{1}{2}+2it\right)\mathds{1}_{\vert t\vert \leq\log^{\epsilon}(q)}(t)dt\right\vert^{2k}  \\  
  \ll \sum_{\chi \in X_q^*}\int_{-\infty}^{\infty}\cdots\int_{-\infty}^{\infty}\left\vert L\left(\frac{1}{2}+2it_1,\chi\right) \cdots  L\left(\frac{1}{2}-2it_{2k},\chi\right)\right\vert \mathds{1}_{||t|| \leq\log^{\epsilon}(q)}(t)dt_1\cdots dt_{2k} \nonumber \\ \nonumber \end{eqnarray}  where $||t||= \max_{i=1,\cdots,2k} \vert t_i\vert$. We will use Theorem \ref{momentshifted} to handle that integral.  In order to do this, we have to control how the shifts $t_i$ are close to each other.

By a permutation change of the variables, we can assume that $t_1\leq t_2\leq \cdots\leq t_{2k}$. Indeed, the integral on $\mathcal{H}$ is equal to $(2k)!$ times the integral with this additional restriction. \\

For every $2k$- tuple $\overline{t}=(t_1,\cdots,t_{2k})$, define a $(2k-1)$- tuple $\overline{j}=(j_1,\cdots, j_{2k-1})$ where $j_{i}= \min\{ i+1\leq j\leq 2k,  \vert t_i-t_j\vert >\frac{1}{\log q}\}$. If for some $i$, no such $j$ exists, we set $j_i=2k+1$. In the following, we will say that $\overline{t}=(t_1,\cdots,t_{2k})$ is of type $\overline{j}$. Let us give few remarks about that definition.  First of all, we have to think about $j_i$ as the first occurrence of a shift lying far from $t_i$. Furthermore, notice that $2\leq j_1\leq j_2 \leq \cdots \leq j_{2k-1} \leq 2k+1$ and that we can split the domain of integration $\mathcal{H}$ in a disjoint union $\mathcal{H}=\cup \mathcal{H}_{\overline{j}}$ of $2k$-tuples $\overline{t}=(t_1,\cdots,t_{2k})$ where $\overline{t}$ is of type $\overline{j}$. Hence, proving Theorem \ref{upperGRH} reduces to bound the contribution of the integral over $\overline{t}$ of type $\overline{j}$ for all possible choices of $\overline{j}$. \\

 The strategy is to apply Theorem \ref{shiftedmoments} in a appropriate way to obtain the expected bound. Using Theorem \ref{shiftedmoments}, we get that the contribution in (\ref{hypercube}) of $\overline{t}$ of type $\overline{j}$ is bounded by

\begin{equation}\label{integral shift} \phi(q) (\log q)^{k/2+\epsilon} \int\cdots\int_{\mathcal{H}_{\overline{j}}} \,\, \left(\prod_{i<j} E_{i,j}\right) \,dt_1 \cdots dt_{2k} \end{equation} where $E_{i,j}$ is defined in Theorem \ref{shiftedmoments}. For every $i=1,2,\cdots,2k-1$, we will essentially bound $\displaystyle{\prod_{j=i+1}^{2k}E_{i,j}}$ in two different ways, depending on whether the variable $t_i$ possesses a close shift or not.

\vspace{2mm} $\bullet$ \underline{Case 1}: Close shifts.  \vspace{1mm} 

 If $t_i$ admits a close shift then $j_i>i+1$. Using the first case of Theorem \ref{shiftedmoments}, we have the following trivial bound 
 
 \begin{equation}\label{closeshifts}\left\vert \prod_{j=i+1}^{2k}E_{i,j} \right\vert \leq (\log q)^{\frac{2k-i}{2}}.\end{equation}

 $\bullet$ \underline{Case 2}: Isolated shifts.  \vspace{1mm} 
 
For those indices $i$, $t_i$ does not admit a close shift, which means that $j_i=i+1$. We remark that $\frac{1}{\vert t_i - t_j\vert} \leq \frac{1}{\vert t_i - t_{j_i}\vert}$ for $j\geq j_i$, since we have $t_1\leq t_2\leq \cdots\leq t_{2k}$. Hence, using again both cases of Theorem \ref{shiftedmoments}, we derive the following bound: 

\begin{equation}\label{isolated} \left\vert \prod_{j=i+1}^{2k}E_{i,j} \right\vert \leq \frac{1}{\vert t_i-t_{i+1}\vert^{\frac{(2k-i)}{2}}}(\log\log q)^{\frac{2k-i}{2}}. \end{equation} To deal with the integral in (\ref{integral shift}), we can make the following linear change of variables:

\begin{equation}\label{changevar}u_i=\begin{cases}t_i-t_{i+1} \,\,\,\text{         if   } i\leq 2k-1,  \\
t_{2k} \hspace{1cm}\text{   if   } i=2k. \\
  \end{cases}\end{equation} Thus, the determinant of the Jacobian being equal to $1$, the integral in (\ref{integral shift}) becomes
  
\begin{equation}\label{afterchange}(\log\log q)^{k(2k-1)}\prod_{i=1\atop i\vert j_i\neq i+1}^{2k-1}(\log q)^{\frac{2k-i}{2}} \int\cdots\int_{\mathcal{D}_{\overline{j}}}\prod_{i=1\atop i \vert j_i=i+1}^{2k-1}\frac{1}{\vert u_i\vert^{\frac{(2k-i)}{2}}}du_1\cdots du_{2k}  \end{equation} where the domain $\mathcal{D}_{\overline{j}}$ is included in 

$$\prod_{i\vert j_i\neq i+1} ] -1/ \log q, 0] \prod_{i\vert j_i=i+1} ] -\log^{\epsilon} q, -1/ \log q].$$ For those $i$ such that $j_i \neq i+1$, we bound the integral over $u_i$ by the length of the interval of integration $1/\log q$. For the other indices, we integrate explicitly on $] -\log^{\epsilon} q, -1/ \log q]$. In order to obtain the expected bound, we need to ``save" a logarithm for each integration $du_i$ for $i=1\cdots 2k-1$. An additional problem arises when the variable does not admit a close shift and we integrate $u_{2k-1}^{-1/2}$. Let us first treat the easiest case.

 \vspace{4mm}
 $\spadesuit \hspace{3mm}$ \underline{Subcase 1}: $j_{2k-1}\neq 2k$.  \vspace{2mm} 
 
  In that case, all the exponents in the denominator of the integral in (\ref{afterchange}) are greater than $1$. Therefore, we obtain after explicit integration that $(\ref{afterchange})$ is bounded by

  \begin{equation}\label{winlog} (\log\log q)^{k(2k-1)}\prod_{i\vert j_i\neq i+1} \frac{(\log q)^{\frac{2k-i}{2}}                                                 }{\log q}\prod_{i=1\atop i \vert j_i=i+1}^{2k-1}\frac{(\log q)^{\frac{2k-i}{2}}                                                 }{\log q} \log\log q \end{equation} where the factor $\log \log q$ comes from the possible integration of $1/u$ when $i=2k-2$. Hence, (\ref{integral shift}) is bounded by 
  
  $$ \phi(q) (\log q)^{f(k)+2k^2\epsilon}$$ where 
  
  \begin{align*} f(k) &=\frac{k}{2}+ \frac{1}{2}  \sum_{i=1}^{2k-1} (2k-i-2) \\
                    & = \frac{k}{2}+ \frac{1}{2} \sum_{i=-1}^{2k-3} i \\
                    & = \frac{k}{2}+ \frac{(2k-3)(2k-2)}{4} -\frac{1}{2} \\
                    & = (k-1)^2
   \end{align*} which proves Theorem \ref{upperGRH} in that case.

\vspace{1mm} 

$\spadesuit \hspace{3mm}$ \underline{Subcase 2}: $j_{2k-1}= 2k$.  \vspace{2mm} 

 The only remaining problem arises when $t_{2k-1}$ does not have a close shift. In that case, an explicit integration in (\ref{afterchange}) is not sufficient to save $\log q$ after integration, but only saves $\log^{1/2} q$. We are going to split the proof in two subcases depending on whether $t_1$ admits a close shift or not.
 
 \vspace{1mm}
 
$\clubsuit \hspace{3mm}$ \underline{Subsubcase 1}: $\fbox{$j_1 \neq 2$}$ 
\vspace{1mm}

 We will use exactly the same bounds as before except for $i=1$. The trivial inequality $\vert t_1-t_{2k}\vert^{-1/2} \leq \vert t_{2k-1}-t_{2k}\vert^{-1/2}$ together with the simple observation that $\vert t_{2k-1}-t_{2k}\vert$ is large (by hypothesis $j_{2k-1}=2k$) implies the following bound 
 
 \begin{equation}\label{modifE1}\prod_{j> 1} E_{1,j} \leq (\log q)^{\frac{2k-2}{2}}\frac{\log \log q}{\vert t_{2k-1}-t_{2k}\vert^{1/2}}.\end{equation} Doing the same change of variables as before in (\ref{changevar}) and using (\ref{modifE1}), we end up with the bound
 
 \begin{equation*}\label{afterchangebis}(\log \log q)^{k(2k-1)}(\log q)^{\frac{2k-2}{2}}\prod_{i=2\atop i\vert j_i\neq i+1}^{2k-2}(\log q)^{\frac{2k-i}{2}} \int\cdots\int_{\mathcal{D}_{\overline{j}}}\prod_{2=1\atop i \vert j_i=i+1}^{2k-2}\frac{1}{\vert u_i\vert^{\frac{(2k-i)}{2}}}\frac{1}{\vert u_{2k-1}\vert}du_1\cdots du_{2k}.\end{equation*} A slightly modification of the computation following (\ref{winlog}) enables us to obtain the expected bound $(\log q)^{(k-1)^2+\epsilon}$. 

 \vspace{1mm}

$\clubsuit \hspace{3mm}$ \underline{Subsubcase 2}: $\fbox{$j_1 = 2$}$ 
\vspace{1mm}

We proceed as in the previous subsubcase with the following bound (the $\log\log$ factor coming from the possible case where the shifts are far away from each other)

$$ \prod_{j> 1} E_{1,j}\ll \frac{(\log\log q)^{\frac{2k-1}{2}}}{\displaystyle{\prod_{j=2}^{2k}\vert t_1-t_j \vert^{1/2}}} \leq \frac{(\log\log q)^{\frac{2k-1}{2}}}{\vert t_1-t_2\vert^{\frac{2k-2}{2}}} \frac{1}{\vert t_1-t_{2k}\vert^{1/2}} \leq \frac{(\log\log q)^{\frac{2k-1}{2}}}{\vert t_1-t_2\vert^{\frac{2k-2}{2}}} \frac{1}{\vert t_{2k-1}-t_{2k}\vert^{1/2}}    $$

Doing the same change of variables as before in (\ref{changevar}), we end up with the same integral as in the previous subsubcase. Hence, the same computation works and this concludes the proof of Theorem \ref{upperGRH}.

  \vspace{3mm}

  \section*{Acknowledgements}

The author is grateful to Dimitri Dias and Igor E. Shparlinski for very helpful discussions and to Maksym Radziwill for pointing him the reference \cite{ChandeeShifts}.

\end{document}